\documentclass{amsart}
\usepackage{amssymb,amstext,amsmath,amscd,amsthm,amsfonts,enumerate,graphicx,latexsym,stmaryrd,multicol}

\usepackage[usenames]{color}
\usepackage[all]{xy}
\usepackage[colorlinks]{hyperref}
\usepackage{cleveref}
\usepackage{tikz-cd}  
\usepackage{tikz} 

\newtheorem{thm}{Theorem}[section]
\newtheorem{lem}[thm]{Lemma}
\newtheorem{prop}[thm]{Proposition}
\newtheorem{cor}[thm]{Corollary}
\theoremstyle{definition}
\newtheorem{dfn}[thm]{Definition}

\newtheorem{chunk}[thm]{}  
\newtheorem{rem}[thm]{Remark}

\newtheorem{notation}[thm]{Notation}
\theoremstyle{remark}

\numberwithin{equation}{thm}
\def\add{\operatorname{add}}

\def\C{\mathcal{C}}

\def\QCoh{\operatorname{QCoh}}

\def\spec{\operatorname{Spec}}

\def \L {\mathbf{L}}

\def\tr{\operatorname{tr}}

\def\X{\mathcal{X}}

\def \R {\mathcal{R}} 
\def \L {\mathcal{L}}

\def \D {\operatorname{D}}
\def \Mod {\operatorname{Mod}}
\begin{document}
\title{On splitting of morphisms induced by unit map of adjoint functors}
\author{Souvik Dey}
\address{S.~Dey,
Faculty of Mathematics and Physics,
Department of Algebra,
Charles University, 
Sokolovsk\'{a} 83, 186 75 Praha, 
Czech Republic}
\email{souvik.dey@matfyz.cuni.cz} 

\begin{abstract}
    Given a right adjoint functor  between triangulated categories and an object in the target category, we show that the unit map of adjunction on that object is a split monomorphism if and only if the object belongs to the additive closure of (all possible) shifts of an object in the image of the functor. Applications to geometric context related to (derived) splinters and rational singularities are given. 
\end{abstract}  

\maketitle

\section{Introduction} 

"Natural" vs. "unnatural" splitting is a recurring scenario of many branches of homological algebra. For instance, consider a short exact sequence of modules $0\to M \to M\oplus N\to N\to 0$ over a commutative Noetherian ring. In general, although this sequence may look split, but it can indeed fail to be split, however, under the additional hypothesis of $M,N$ being finitely generated, it does hold that such an exact sequence is indeed split (\cite{miyata}). In this article, we are concerned with another kind of splitting phenomena. Namely, let $\R: \mathcal T \to \mathcal S$ be a functor between additive categories admitting a left adjoint $\L:\mathcal S \to \mathcal T$. By definition, these come with unit $\eta: 1_{\mathcal S} \to \R\L$ and counit $\varepsilon: \L\R \to 1_{\mathcal S}$ transformations. Our main starting point of this article is to understand: given $M\in \mathcal S$, when is the natural map $\eta_M:M\to \R\L(M)$ split? This is in particular motivated by the following scenario: Let $f: S\to X$ be a morphism of schemes, then the derived pushforward $\mathbf Rf_*: \D(\QCoh S)\to \D(\QCoh X)$ between the derived categories of quasi-coherent sheaves is a right-adjoint to the derived pullback $\mathbf Lf^*: \D(\QCoh X)\to \D(\QCoh S)$ and the ordinary pushforward $f_*: \QCoh S\to \QCoh X$ is a right-adjoint to the  pullback $f^*: \QCoh X \to \QCoh S$. Remembering that $\mathbf Lf^* \mathcal O_X=f^*\mathcal O_X=\mathcal O_S$, one then has the following naturally associated concepts coming from splitting of the unit map of adjunction

\begin{dfn}(\cite[Definition 1.3]{bhatt}) A scheme $S$ is called a derived splinter, or simply a $D$-splinter, if for any proper surjective map $f: X\to S$, the natural  map $\mathcal O_S\to \mathbf Rf_*\mathcal O_X$ splits. 
\end{dfn}

\begin{dfn}(\cite[Definition 1.2]{bhatt}) A scheme $S$ is called a splinter, if for any finite surjective map $f: X\to S$, the natural  map $\mathcal O_S\to f_*\mathcal O_X$ splits. 
\end{dfn}

The importance of the definition of derived splinter arises from the fact that for schemes of finite type over a field of characteristic $0$, derived splinters are exactly rational singularities (\cite[Theorem 2.12]{bhatt}). 

 If $X$ is a normal integral scheme, then the existence of  a splitting $\mathcal O_S\to f_*\mathcal O_X$ for a morphism $f:X \to S$ implies $S$ is also normal integral by Proposition \ref{splnorm} (compare with the first half of \cite[Example 2.1]{bhatt}). 

Our main technical result, Lemma \ref{wellknown}, says that as soon as $M\in \mathcal S$ is a direct summand of a finite direct sum of copies of an object in the image of $\R : \mathcal T\to \mathcal S$, then the natural map $\eta_M:M\to \R\L(M)$ is a split monomorphism. When $\mathcal T$ and $\mathcal S$ are triangulated categories and the functor $\R$ commutes with shift, then it is also enough to assume that $M$  is a direct summand of a finite direct sum of copies of shifts of an object in the image of $\R : \mathcal T\to \mathcal S$, see Proposition \ref{levelsplit}. As  one of our  applications, we deduce the following result which is motivated by the notions of splinters and derived splinters.

\begin{thm}[\Cref{splinter}]\label{first} Let $f: X\to S$ be a morphism of Noetherian schemes. Assume that there exists $B\in \operatorname{D}(\operatorname{QCoh} X)$ such that $\mathcal O_S \in \langle \mathbf R f_* B\rangle_1 $. Then, the following are true:

\begin{enumerate}[\rm(1)]
    \item The natural map $\mathcal O_S \to \mathbf R f_* \mathcal O_X$ is a split monomorphism.

    \item If $f$ is an affine morphism and $X$ is an integral normal scheme, then $S$ is an integral normal scheme. 
    \end{enumerate}
    
\end{thm}

We refer the reader to Definition \ref{def:thick_subcategory_triangulated} for any undefined notation(s) used in \Cref{first}.   

As an immediate consequence of this,  Proposition \ref{splnorm}, and the proof of \cite[Theorem 2.12]{bhatt}, we are able to say the following:  

\begin{thm}[Corollary \ref{dsplnorm}]\label{1.3} Let $S$ be a scheme of finite type over a field of characteristic $0$. If there exists a resolution of singularities $f: X\to S$ such that $\mathcal O_S$ is a direct summand of a finite direct sum of copies of shifts of $\mathbf R f_* B$  for some $B\in \operatorname{D}(\operatorname{QCoh} X)$, then $S$ has rational singularities. 
\end{thm}

Needless to say that the converse  Theorem \ref{1.3} is immediate by the definition of schemes having rational singularities.

For the entirety of this article, the following notations have been fixed.

\begin{notation}
    Let $X$ be a Noetherian scheme. We will consider the following triangulated categories:
    \begin{enumerate}
        \item $\D(X)$ is the unbounded derived category of complexes of $\mathcal{O}_X$-modules
        \item $\D(\operatorname{Qcoh} X)$ is the unbounded derived category of complexes of $\mathcal{O}_X$-modules with quasi-coherent cohomology
        \item $\D^b({\operatorname{Coh } } X)$ and $\D^b({\operatorname{QCoh } } X)$ are the derived categories of bounded complexes of $\mathcal{O}_X$-modules with coherent and quasi-coherent cohomology, respectively.  

\item When $R$ is a commutative Noetherian ring and $X=\spec(R)$ is affine, we can naturally identify $\text{QCoh } X$ with $\text{Mod } R$, the abelian category of all $R$-modules, and $\text{Coh } X$ with $\text{mod } R$, the abelian category of finitely generated $R$-modules. 

\item Given any ring homomorphism of commutative Noetherian rings $f: R\to S$, we identify it with the induced morphism of schemes $f: \spec(S)\to \spec(R)$. 
        
    \end{enumerate}
\end{notation}

\section{preliminaries}

In this section, we briefly recall the notion of standard building process in triangulated categories. Some of the primary sources of references are  \cite{BVdB:2003, Rouquier:2008}.  

Let $\mathcal{T}$ be a triangulated category with shift functor $[1]\colon \mathcal{T} \to \mathcal{T}$ and $\mathcal{C}$ be a subcategory of $\mathcal{T}$.

\begin{dfn}\label{def:thick_subcategory_triangulated}
    \begin{enumerate}
        \item A triangulated subcategory $\mathcal C$ of $\mathcal{T}$ is \textbf{thick} if it is closed under direct summands. The smallest thick subcategory of $\mathcal{T}$ containing $\mathcal{C}$ is denoted $\langle \mathcal{C} \rangle$.
        \item Consider the following additive subcategories of $\mathcal{T}$:
        \begin{enumerate}
            \item $\operatorname{add}^{\Sigma}(\mathcal{C})$ is the strictly full subcategory of retracts of finite direct sum of shifts of objects in $\mathcal{C}$
            \item $\langle \mathcal{C} \rangle_0$ consists of all objects in $\mathcal{T}$ isomorphic to the zero object
            \item $\langle \mathcal{C} \rangle_1 :=\operatorname{add}^{\Sigma}(\mathcal{C})$
            
            \item $\langle \mathcal C \rangle_n := \operatorname{add}^{\Sigma}\{ M \mid \text{there exists an exact triangle } L\to M\to N\to  L[1] \text{ with } L\in \langle \mathcal{C} \rangle_{n-1}, \text{ and } N\in \langle \mathcal{C} \rangle_1 \}$ if $n\geq 2$.
        \end{enumerate}
    \end{enumerate}
\end{dfn} 

For the basics of categories and adjoint functors, we refer the reader to any text on category theory, for instance \cite[Chapter IV]{mac}.  




\section{main technical results}

For an additive category $\C$ and a subcategory $\X\subseteq \C$, $\add_{\C} \X$ will denote the smallest subcategory of $\C$ that contains $\X$ and is closed under finite direct sums and direct summands. When the ambient category $\C$ is clear from context, we drop the subscript $\C$.

It is well-known that adjoint functors between additive categories are additive. 

\begin{lem}\label{wellknown} Let $\mathcal S, \mathcal T$ be additive categories. Let $\R: \mathcal T \to \mathcal S$ be a functor admitting a left adjoint $\L:\mathcal S \to \mathcal T$. Let $\varepsilon: \L\R \to 1_{\mathcal T}$ and $\eta: 1_{\mathcal S}\to \R\L$ be the counit and unit respectively. Then, the following are true:

\begin{enumerate}[\rm(1)]
    \item If there exists objects $M,N$ in $\mathcal S$ and $\mathcal T$ respectively such that $M\in \add \left(\R(N)\right)$, then the natural  map $\eta_M: M \to \R\L(M)$ is a split monomorphism.

    \item  If there exists objects $M,N$ in $\mathcal S$ and $\mathcal T$ respectively such that $N\in \add \left(\L(M)\right)$, then   the natural map $\varepsilon_N: \L\R(N) \to N$ is a split epimorphism.  
\end{enumerate}

\end{lem}

\begin{proof} (1) Since $R$ is an additive functor, hence we may assume that $M$ is a retract of $\R(N)$ for a single object $N\in \mathcal T$.  Thus, there exists morphisms $f: M\to \R(N)$ and $g: \R(N)\to M$ such that $g\circ f=id_M$. By definition of adjunction, the composition $\R\xrightarrow{\eta \R} \R\L \R \xrightarrow{\R\varepsilon}\R$ is $1_{\R}$. We have the following diagram 

\begin{displaymath} 
\begin{tikzcd}
M \arrow[r, "f"] \arrow[d, "\eta_M"'] & \R(N) \arrow[d, "\eta_{\R(N)}"']                     \\
\R\L(M) \arrow[r, "\R\L(f)"']             & \R\L \R(N) \arrow[u, "\R(\varepsilon_N)"', shift right]
\end{tikzcd}
\end{displaymath}

By naturality of $\eta: 1_{\mathcal S} \to \R\L$, we know $\R\L(f)\circ \eta_M = \eta_{\R(N)}\circ f$. Since the composition $\R\xrightarrow{\eta \R} \R\L \R \xrightarrow{\R\varepsilon} \R$ is $1_{\R}$, hence $\R(\varepsilon_N)\circ \eta_{\R(N)}=Id_{\R(N)}$. Hence, $$g\circ \R(\varepsilon_N)\circ \R\L(f)\circ \eta_M=g\circ \R(\varepsilon_N)\circ \eta_{\R(N)}\circ f=g\circ Id_{\R(N)}\circ f=g\circ f=id_M$$ This proves our claim. 

(2) Is dual to (1). Indeed, since $\L$ is an additive functor, hence we may assume that $N$ is a retract of $\L(M)$ for a single object $M\in \mathcal S$.  Thus, If $h: N\to \L(M)$ and $j: \L(M)\to N$ are morphisms such that $j\circ h=id_N$, then $\varepsilon_N\circ \L\R(j)\circ \L(\eta_M)\circ h=id_N$. 
\end{proof} 

\begin{rem} The author is unaware if there is a variation of Lemma \ref{wellknown}(1) where one assumes $\R(N)\in \add(M)$ and concludes a splitting of $\eta_M: M\to \R\L(M)$. 
\end{rem}

Next, we aim to give some consequences of Lemma \ref{wellknown} to triangulated categories. First we need an elementary lemma which is probably well known, but we add a proof for the convenience of the reader.

\begin{lem}\label{standard} Let $\mathcal S, \mathcal T$ be triangulated categories. Let $\R: \mathcal T \to \mathcal S$ be a functor admitting a left adjoint $\L :\mathcal S \to \mathcal T$.  If $\R$ commutes with shift then so does $\L$. 
\end{lem}

\begin{proof}  Let $M\in \mathcal S$. We have the following isomorphisms of functors using adjunction properties of $\R$ and $\L$ and properties of the shift functor. 

\begin{align}
\operatorname{Hom}_{\mathcal{T}}\left(\L(M[n]),-\right)
&\cong\operatorname{Hom}_{\mathcal{S}}\left(M[n],\R(-)\right)\\
&\cong\operatorname{Hom}_{\mathcal{S}}\left(M,\R(-)[-n]\right)\\
&\cong\operatorname{Hom}_{\mathcal{S}}\left(M,\R(-)[-n])\right)\\
&\cong\operatorname{Hom}_{\mathcal{T}}\left(\L(M),(-)[-n]\right)\\
&\cong\operatorname{Hom}_{\mathcal{T}}\left(\L(M)[n],-\right)
\end{align}

By Yoneda's Lemma, we then conclude $\L(M[n])\cong \L(M)[n]$. 
\end{proof} 

Given a collection of objects $\mathcal C$ in a triangulated category $\mathcal T$, the subcategories $\langle \mathcal C \rangle_i$ were defined in \cite[3.1.1]{Rouquier:2008}, \cite[Section 2.2]{BVdB:2003}. For $i=1$, $\langle \mathcal C \rangle_1$ is nothing but the full subcategory of $\mathcal T$ consisting of all objects that are direct summands of shifts of finite direct sum of copies of objects from $\mathcal C$.  

\begin{prop}\label{levelsplit} Let $\mathcal S, \mathcal T$ be triangulated categories. Let $\R:\mathcal T \to \mathcal S$ be a functor commuting with sifts. Assume $\R$ admits a left adjoint $\L:\mathcal S \to \mathcal T$. Let $\varepsilon: \L\R \to 1_{\mathcal T}$ and $\eta: 1_{\mathcal S}\to \R\L$ be the counit and unit respectively. 
Then, the following are true:

\begin{enumerate}[\rm(1)]
    \item If there exists objects $M,B$ in $\mathcal S$ and $\mathcal T$ respectively such that $M\in \langle \R(B) \rangle_1$, then the natural  map $\eta_M : M \to \R\L(M)$ is a split monomorphism. 

    \item If there exists objects $M, B$ in $\mathcal S$ and $\mathcal T$ respectively such that $B\in \langle \L(M) \rangle_1$, then the natural  map $\varepsilon_B: \L\R(B)\to B$ is a split epimorphism.  
\end{enumerate}

\end{prop}

\begin{proof} (1) By hypothesis, $M$ is a retract of finite direct sum of copies of shifts of $\R(B)$. Since $\R$ is an additive functor commuting with sift, hence $M$ is a retract of $\R(N)$ for some object $N$ in $\mathcal T$. Now the claim is immediate from Lemma \ref{wellknown}(1).  

(2)  Is dual to (1) (in view of Lemma \ref{standard}) and obtained similarly using  Lemma \ref{wellknown}(2).  
\end{proof}

\section{some applications}

We start by giving some immediate consequences of Lemma \ref{wellknown} when applied to some abelian categories. 

\begin{prop}\label{restriction} Let $f: R\to S$ be a morphism of commutative rings. If $M\in \Mod R$ and $N\in \Mod S$ are such that $M\in \add(f_* N)$, then the natural map $M\to f_*(M\otimes_R S)$ is a split monomorphism.   
\end{prop}

\begin{proof}  As  the restriction of scalars  $f_*:\Mod S \to \Mod R$ has a left adjoint $(-)\otimes_R S: \Mod R \to \Mod S$, the claim follows from Lemma \ref{wellknown}(1).  
\end{proof}

For the definition of trace ideal $\tr(-)$ of a module, we refer the reader to \cite[Section 2]{lin}. 

\begin{cor} Let $f: R\to S$ be a finite morphism of commutative Noetherian rings. If $\tr_R(f_*S)=R$, then the natural map $R\to f_* S$ is a split monomorphism.
\end{cor}

\begin{proof} Since $f$ is a finite morphism, $f_*S$ is a finitely generated $R$-module. In view of \cite[Proposition 2.8(iii)]{lin}, the claim now follows from Proposition \ref{restriction}.   
\end{proof}



The following Proposition \ref{splint} is interesting due to the notion of splinter \cite[Definition 1.2]{bhatt}. 

\begin{prop}\label{splint} Let $f: X\to S$ be a  morphism of Noetherian schemes. If there exists $N\in \operatorname{QCoh} X$ such that $\mathcal O_S \in \add(f_* N)$, then the natural map $\mathcal O_S \to f_*\mathcal O_X$ is a split monomorphism.
\end{prop}   

\begin{proof}


That $f^*: \QCoh(S)\to \QCoh(X)$ is left adjoint to $f_*: \QCoh(X)\to \QCoh(S)$ follows from \cite[Proposition 7.11.]{grtz1}. Since $f^*\mathcal O_S=\mathcal O_X$, hence the claim is immediate from Lemma \ref{wellknown}(1).  
\end{proof}

\begin{lem}\label{sumnorm} Let $B$ be a normal domain. Let $A\subseteq B$ be a subring such that there exists an $A$-linear map $\phi: B\to A$ such that $\phi|_A=id_A$. Then, $A$ is a normal domain.  
\end{lem}

\begin{proof} Let $a/b\in Q(A)\subseteq Q(B)$ be integral over $A$. Then, $a/b$ is integral over $B$, hence $a/b=s$ for some $s\in B$ as $B$ is normal. Thus, $a=bs$. Applying $\phi$, we get $a=b\phi(s)$. Thus, $bs=a=b\phi(s)$, hence $s=\phi(s)$ as $B$ is an integral domain. Thus, $a/b=\phi(s)\in A$, concluding our proof.  
\end{proof}

\begin{prop}\label{splnorm} Let $f: X\to S$ be a  morphism of Noetherian schemes. If $X$ is a normal integral scheme  and there exists $N\in \operatorname{QCoh} X$ such that $\mathcal O_S \in \add(f_* N)$, then $S$ is a normal integral scheme. 
\end{prop}

\begin{proof} By Proposition \ref{splint}, the natural map $\mathcal O_S \to f_*\mathcal O_X$ is a split monomorphism. Being normal is an open condition (\cite[Tag 033J]{stacks-project}). Moreover,  $X$ is a normal and integral implies $\mathcal O_X(U)$ is a normal domain for every open subset $U$ of $X$ (\cite[Tag 033J, Tag 0358]{stacks-project}). Hence, we may assume both $S$ and $X$ are affine, and then we are done by Lemma \ref{sumnorm}.
\end{proof}

Next we provide some applications of  Proposition \ref{levelsplit} to certain derived categories. 
The following Corollary \ref{splinter} is interesting in view of the notion of derived splinter (also known as D-splinter) \cite[Definition 1.3]{bhatt}. 

\begin{cor}\label{splinter} Let $f: X\to S$ be a morphism of Noetherian schemes. Assume that there exists $B\in \operatorname{D}(\operatorname{QCoh} X)$ such that $\mathcal O_S \in \langle \mathbf R f_* B\rangle_1 $. Then, the following are true:

\begin{enumerate}[\rm(1)]
    \item The natural map $\mathcal O_S \to \mathbf R f_* \mathcal O_X$ is a split monomorphism.

    \item If $f$ is an affine morphism and $X$ is an integral normal scheme, then $S$ is an integral normal scheme.  
\end{enumerate}

\end{cor} 

\begin{proof}


We know that $\mathbf L f^*: \operatorname{D}(\operatorname{\QCoh} S)\to \operatorname{D}(\operatorname{\QCoh} X)$ is left adjoint to $\mathbf R f_*:\operatorname{D}(\operatorname{\QCoh} X)\to \operatorname{D}(\operatorname{\QCoh} S)$ (\cite[Chapter 3, Proposition 3.2.3]{lip}). 

(1) Since $\mathbf Lf^* \mathcal O_S=f^*\mathcal O_S=\mathcal O_X$, the claim is now immediate from Proposition \ref{levelsplit}(1).  

(2) Since $f$ is an affine morphism, hence $\mathbf Rf_* B\simeq f_* B$ (\cite[Tag 0G9R]{stacks-project}). Now we are done by part (1) and \Cref{splnorm}. 
\end{proof}

In view of the definition of derived splinter,  \Cref{splinter}(1) unifies and highly generalizes the non-trivial implications of \cite[Lemma 3.11, 3.12, Theorem 3.13]{lank2024triangulated} by removing all assumptions from the schemes and the morphism, as well as weakening the hypothesis $\mathcal O_S \in \langle \mathbf R f_* \mathcal O_X\rangle_1 $ to $\mathcal O_S \in \langle \mathbf R f_* B\rangle_1 $ for some arbitrary $B\in \D(\text{QCoh } X)$. Similarly, since finite morphisms are affine, \Cref{splinter}(2) highly generalizes the non-trivial implication of \cite[Theorem B]{lank2024triangulated}.

    

\begin{cor}\label{dsplnorm} Let $S$ be a variety of characteristic $0$. If there exists a resolution of singularities $f: X\to S$ such that $\mathcal O_S \in \langle \mathbf R f_* B\rangle_1 $  for some $B\in \operatorname{D}(\operatorname{QCoh} X)$, then $S$ has rational singularities. 
\end{cor}

\begin{proof} From Corollary \ref{splinter} it follows that the natural morphism $\mathcal O_S \to \mathbf Rf_* \mathcal O_X$ is a split monomorphism. It then follows that $S$ is normal by Proposition \ref{splnorm}. Now our claim follows from the proof of \cite[Theorem 2.12]{bhatt}. 
\end{proof}

Similar to \Cref{splinter}(1), \Cref{dsplnorm} highly generalizes the non-trivial implication of \cite[Proposition 3.15]{lank2024triangulated}.

\begin{cor}\label{newdsplnorm} Let $X, S$ be  schemes of characteristic $0$. If there exists a surjective morphism of schemes $f: X\to S$ such that $X$ is integral, has rational singularities, and $\mathcal O_S \in \langle \mathbf R f_* B\rangle_1 $  for some $B\in \operatorname{D}(\operatorname{QCoh} X)$, then $S$ has rational singularities. 
\end{cor}

\begin{proof} From Corollary \ref{splinter} it follows that the natural morphism $\mathcal O_S \to \mathbf Rf_* \mathcal O_X$ is a split monomorphism. We are now done by \cite[Theorem 9.2]{murayama2024relative}. 
\end{proof}

\Cref{newdsplnorm} highly generalizes \cite[Theorem 3.14(2)$\implies$(1)]{lank2024triangulated}

\begin{chunk}
A variety $X$ over a field of characteristic zero is said to have Du Bois singularities if the natural map
$\mathcal O_X \to \underline{\Omega}^0_X$ splits \cite{sch, kov}, where $\underline{\Omega}^0_X$ denotes the 0-th graded piece of the DuBois complex of $X$. 
\end{chunk}

A proof similar to \cite[Theorem 3.16(2)$\implies $(1)]{lank2024triangulated} using \Cref{splinter} instead of \cite[Lemma 3.11]{lank2024triangulated} produces the following

\begin{cor} Let $S$ be an embeddable variety of characteristic $0$, i.e., there exists a closed immersion $i:S\to Y$ for a smooth variety $Y$. If $\mathcal O_S\in \langle \underline{\Omega}^0_X\rangle_1$, then $S$ is a DuBois singularity. 
\end{cor}

\bibliographystyle{plain}
\bibliography{mainbib}

\end{document}